\documentclass%%[draft]
{article}
\usepackage[T2A]{fontenc}
\usepackage[cp1251]{inputenc}
\usepackage[english,russian]{babel}
\usepackage[tbtags]{amsmath}
\usepackage{amsfonts,amssymb,mathrsfs,amscd}
\usepackage{wasysym}
\usepackage{verbatim}
\usepackage{eucal}
\usepackage{upgreek}
\usepackage{graphicx}
\usepackage{enumerate}
\let\No=\textnumero

\usepackage[hyper]{msb-a}
\JournalName{}
\numberwithin{equation}{section}
\theoremstyle{plain}
\newtheorem{theorem}{Теорема}

\newtheorem{propos}{Предложение}
\newtheorem{corollary}{Следствие}
\newtheorem{lemma}{Лемма}

\newtheorem{thF}{Теорема Фростмана}

\theoremstyle{definition}
\newtheorem{proof}{Доказательство}
\newtheorem{remark}{Замечание}
\newtheorem{definition}{Определение}
\newtheorem{example}{Пример}

\renewcommand{\leq}{\leqslant} 
\renewcommand{\geq}{\geqslant}
\newcommand{\RR}{\mathbb{R}} 
\newcommand{\CC}{\mathbb{C}} 
\newcommand{\NN}{\mathbb{N}} 
\newcommand{\DD}{\mathbb{D}}
\newcommand{\BB}{\mathbb{B}}  

\renewcommand{\Re}{{\mathrm{Re}\,}}
\renewcommand{\Im}{{\mathrm{Im}\,}}

\DeclareMathOperator{\supp}{{\sf supp}}
\DeclareMathOperator{\dd}{\,{\mathrm d\!}}

\DeclareMathOperator{\rad}{{\text{\tiny \rm rd}}}

\begin{document} 
\title{Интегралы  от разностей субгармонических  функций.~III. Мера и обхват Хаусдорфа, интегрирование по липшицевым кривым и поверхностям}
	
\author[B.\,N.~Khabibullin]{Б.\,Н.~Хабибуллин}
\address{Башкирский государственный университет}
\email{khabib-bulat@mail.ru}

\date{15.07.2021}
\udk{517.547.2 : 517.574 : 517.518.1}

 \maketitle

\begin{fulltext}

\begin{abstract}  
Получены дополнительные интегральные неравенства для интегралов от разностей субгармонических функций по мерам Бореля на шарах в многомерном евклидовом пространстве. Эти интегралы по-прежнему оцениваются сверху через характеристику Неванлинны и различные характеристики  меры Бореля и её носителя.  Все результаты новые и для логарифмов модулей мероморфных функций на кругах в комплексной плоскости.   Допускается интегрирование по мерам Бореля с носителем на фрактальных множествах, а оценки в этих случаях даются через 
меру и обхваты Хаусдорфа носителя меры Бореля. Отдельно отмечены важные в применениях частные случаи функций во  всей комплексной плоскости и пространстве, в единичном круге или  шаре, а также 
 интегрирования по длине на подмножествах липшицевых кривых и по площади на подмножествах липшицевых гиперповерхностей.

Библиография:  14 названия  

Ключевые слова: мероморфная функция, $\delta$-субгармоническая  функция, характеристика Неванлинны, заряд Рисса,  $h$-мера Хаусдорфа, $h$-обхват Хаусдорфа, $p$-мерная мера Хаусдорфа, $p$-мерный обхват Хаусдорфа,  липшицева кривая, лиршицева гиперповерхность, спрямляемость 

%%Integral inequalities are obtained for integrals of differences of subharmonic functions with respect to Borel measures on balls in %%a multidimensional Euclidean space. These integrals are estimated from above through the Nevanlinna characteristic and %%various characteristics of the Borel measure and its support. The main criterion theorem on such estimates is given with several %%equivalent statements of different nature. All the results are also new for logarithms of modules of meromorphic functions on %%discs in the complex plane.  They contain, as very special cases, all the previous results that go back to the classical lemma of %%Edrey-Fuchs on small arcs. Integration with respect to Borel measures with a support on fractal sets is allowed, and estimates in %%these cases are given in terms of the Hausdorff measures and  contents for the support of Borel measure. Special cases %%considered  that are important in applications: for functions to  the entire complex plane and space, in a unit circle or ball, as %%well as integrations in  length on subsets of Lipschitz curves and in area on subsets of Lipschitz hypersurfaces.
%%Bibliography:  41 titles 

%%Key words: meromorphic function, difference of subharmonic functions, Nevanlinna characteristic, 
%%Hausdorff measure, Lipschitz curve

%%\noindent{\bf MSC 2010:} 30D35, 31A05, 31B05 
%%

\end{abstract}

\markright{Интегралы  от разностей субгармонических  функций \dots}

%%\footnotetext[0]{Исследование выполнено при поддержке ????
%% проекта РНФ № .... 
%%в рамках реализации программы развития Научно-образовательного математического  центра Приволжского %%федерального округа, дополнительное соглашение № ???.
%%}
	
Используются прежние определения и обозначения, иногда с напоминаниями, из первых двух частей \cite{Kha21N1}--\cite{Kha21N2} нашей работы.

\section{Снова модуль непрерывности  меры\\ и интегральные неравенства}

Ещё раз напомним определение \cite[определение 2]{Kha21N1} и элементарные свойства модуля непрерывности меры из первой части \cite{Kha21N1} нашей работы.
\begin{definition}\label{defmc}
 {\it Модуль непрерывности\/} меры  Бореля 
$\mu$ на $\RR^{\tt d}$ ---  функция 
\begin{equation}\label{hmuR}
{\sf h}_{\mu}\colon t\underset{t\in \RR^+}{\longmapsto} \sup\limits_{y\in \RR^{\tt d}}\mu \bigl(\overline B_y(t)\bigr)
=\sup\limits_{y\in \RR^{\tt d}}\mu_y^{\rad}(t)\in \overline \RR^+.  
\end{equation}
\end{definition}
\begin{propos}[{\rm \cite[предложение 2]{Kha21N1}}]\label{proh}
Пусть   $\mu$ ---  мера Бореля на $\RR^{\tt d}$ и  
\begin{equation}\label{Mmu}
{\tt M}:=\mu (\RR^{\tt d})\in \overline \RR^+.
\end{equation}  
 Тогда ${\sf h}_{\mu}$  --- возрастающая функция, удовлетворяющая неравенству 
\begin{subequations}\label{hR}
\begin{align}
{\sf h}_{\mu}(t)&\leq {\tt M}\quad\text{при всех $t\in  \RR^+$}, 
\tag{\ref{hR}$\leq$}\label{hRMl}
\\ 
\intertext{а если носитель $\supp \mu$ меры $\mu$ содержится в шаре $\overline B(r)$, то}
{\sf h}_{\mu}(t)&\equiv {\tt M}\quad\text{при всех $t\geq r$}.
\tag{\ref{hR}$\equiv$}\label{h<M}
\end{align}
\end{subequations}
\end{propos}

Следующее небольшое уточнение \cite[основная теорема]{Kha21N1} --- это лёгкое следствие импликации I$\Longrightarrow$II из \cite[теорема-критерий]{Kha21N1}. Уточнение касается прежде всего того, что за счёт некоторого увеличения абсолютных постоянных в неравенстве  \cite[основная теорема, (3.5)]{Kha21N1} удаётся последний интеграл
в \cite[основная теорема, (3.5T)]{Kha21N1} по интервалу $(0,R+r]$ заменить на меньший интеграл по интервалу $(0,r]$, не зависящему от выбора $R>r$.

\begin{theorem}\label{th1} Пусть $0<r\in \RR^+$, а  $\mu$ ---  мера Бореля на замкнутом шаре   $\overline B(r)\subset \RR^{\tt d}$   радиуса $r$ с центром в нуле  
полной меры \eqref{Mmu} с модулем непрерывности ${\sf h}_{\mu}$ из \eqref{hmuR} и 
выполнено условие
\begin{equation}\label{emcont}
\int_0 \frac{{\sf h}_{\mu}(t)}{t^{\tt d-1}}\dd t< +\infty.
\end{equation} 
Тогда любая $\delta$-субгармоническая  функция   $U\not\equiv\pm\infty$ на замкнутом шаре $\overline B(R)$
радиуса  $R>r$  $\mu$-суммируема  и
\begin{equation}\label{URm}
\int_{\overline B(r)} U^+\dd \mu \leq  
 A_{\tt d}(r,R){\boldsymbol  T}_U( r, R)
 \biggl({\tt M}\max\{1, r^{2-\tt d}\}
+\widehat{\tt d}\int_0^{r}\frac{{\sf h}_{\mu}(t)}{t^{{\tt d}-1}}\dd t\biggr),
\end{equation}
с конечной правой частью, где \cite[(2.2A)]{Kha21N2}
\begin{equation}
 A_{\tt d}(r,R):=5\max\bigl\{1, {\tt d}-2\bigr\}\Bigl(\frac{R+r}{R-r}\Bigr)^{\tt d-1}\max\Bigl\{1, (R-r)^{\tt d-2}\Bigr\},
\label{{UR}A}
\end{equation}
а первый аргумент $r$ в ${\boldsymbol  T}_U( r, R)$ можно заменить на любое  число  $r'\in [0,r]$.
\end{theorem}
\begin{proof} По определению \cite[(1.7)]{Kha21N2} радиальной  проинтегрированной считающей функции с центром $y\in \RR^{\tt d}$ для меры $\mu$ очевидны  неравенства
\begin{equation*}
{\sf N}_y^{\mu}(x):=\widehat{\tt d}\int_0^x \frac{\mu_y^{\rad}(t)}{t^{\tt d-1}}\dd t
\overset{\eqref{hmuR}}{\leq} \widehat{\tt d}\int_0^x \frac{{\sf h}_{\mu}(t)}{t^{\tt d-1}}\dd t, 
\quad \widehat {\tt d}:=\max\{1,{\tt d}-1\},
\end{equation*} 
для любого $x\in \RR^+$. Поскольку правая часть здесь не зависит от $y\in \overline B(r)$ и по условию \eqref{emcont} при некотором $x>0$ конечна, то выполнено условие (2.1) утверждения I из \cite[теорема-критерий]{Kha21N2}. При этом по  \cite[лемма 2]{Kha21N2} мера $\mu$ конечна, откуда 
\begin{equation*}
\sup_{y\in B(r)}{\sf N}_y^{\mu}(r)\leq \widehat{\tt d}\int_0^r \frac{{\sf h}_{\mu}(t)}{t^{\tt d-1}}\dd t<+\infty.
\end{equation*} 
По  последнему  неравенству из импликации  {I}$\Longrightarrow${II}  из \cite[теорема-критерий]{Kha21N2} получаем  $\mu$-суммируемость рассматриваемой функции $U$ на $\overline B(R)$ и неравенство \eqref{URm} с учётом \eqref{Mmu} и с конечной правой частью.
\end{proof}

Следующая теорема даёт неравенства с более явными выражениями.

\begin{theorem}\label{th2} Пусть $0<r\in \RR^+$ и $h\colon [0,r]\to \RR^+$ --- непрерывная функция с $h(0)=0$, дифференцируемая на $(0,r)$,  для которой  
\begin{equation}\label{sh}
\frac{1}{{\sf s}_h}:=\inf_{t\in (0,r)}\frac{th'(t)}{h(t)}-({\tt d-2})>0.
\end{equation}
Тогда для любой меры Бореля $\mu$ на $\overline B(r)$ полной меры\/ ${\tt M}:=\mu \bigl(\overline B(r) \bigr)$ и  с модулем непрерывности ${\sf h}_{\mu}$ из \eqref{hmuR}, удовлетворяющим неравенству  
\begin{equation}\label{ch}
{\sf h}_{\mu}(t)\leq h(t)\quad\text{при всех $t\in [0,r]$}, 
\end{equation}
 любая  $\delta$-субгармоническая   функция   $U\not\equiv \pm \infty$ на шаре $\overline B(R)$ радиуса $R>r$
 $\mu$-суммируема, а  для единственного значения  $h^{-1}(M)\leq r$ имеем  неравенство
\begin{subequations}\label{URh}
\begin{align}
\int U^+\dd \mu \leq&  
5\frac{R+r}{R-r} {\boldsymbol  T}_U(r,R) \, {\tt M}\ln\frac{e^{1+{\sf s}_h}r}{h^{-1}({\tt M})}
\quad\text{при ${\tt d}=2$, т.е. для\/ $\CC$,}
\tag{\ref{URh}$\CC$}\label{{URh}T}
\\
\int U^+\dd \mu \leq&  
A_{\tt d}(r,R) {\boldsymbol  T}_U(r,R)\, {\tt M}\biggl(1+\frac{1+({\tt d}-2){\sf s}_h}{\bigl(h^{-1}({\tt M})\bigr)^{\tt d-2}}
\biggr)
\quad\text{при  ${\tt d}>2$,}
\tag{\ref{URh}{\tt d}}\label{{URh}2}
\end{align}
\end{subequations}
где $r$ в ${\boldsymbol  T}_U(r,R)$ из правых  частей  \eqref{{URh}T}и \eqref{{URh}2} 
можно заменить на любое число   $r'\in [0,r]$, а   $A_{\tt d}(r,R)$ в \eqref{{URh}2} --- величина из \eqref{{UR}A}.
\end{theorem}
\begin{proof} Если точная нижняя грань в \eqref{sh} равна $+\infty$, то $h=0$ на $(0,r)$,  
по условию  \eqref{ch} мера $\mu$ нулевая и неравенства \eqref{URh} по принятому 
соглашению $0\cdot \pm\infty:=: \pm\infty\cdot 0:=0$ тривиальны. 
Поэтому далее ${\sf s}_h>0$. Отметим некоторые свойства функции $h$.

Прежде всего из условия \eqref{ch} следует, что производная $h'$ строго положительна на $(0,r)$, откуда   
функция $h$ строго возрастающая на открытом интервале $(0,r)$, а в силу непрерывности {\it строго возрастающая на отрезке\/ $[0,r]$.} В частности, $h(t)>0$ при $t\in (0,r]$.  
Исходя из   определения числа ${\sf s}_h>0$ в  \eqref{ch},
непосредственными вычислениями убеждаемся, что
\begin{equation}\label{htd}
\frac{\dd}{\dd t}\Bigl(\frac{h(t)}{t^{\tt d-2}}\Bigr)=
\Bigl(\frac{th'(t)}{h(t)}+{\tt 2-d}\Bigr)\frac{h(t)}{t^{\tt d-1}}\overset{\eqref{ch}}{\geq}
\frac{1}{{\sf s}_h}\frac{h(t)}{t^{\tt d-1}}>0
\quad\text{при всех $t\in (0,r)$}.
\end{equation}
и   функция $t\mapsto h(t)/t^{\tt d-2}$  строго возрастающая   на $(0,r]$. Продолжение по непрерывности этой  функции в точку\/ $0$   по непрерывности как $\lim\limits_{0<t\to 0}h(t)/t^{\tt d-2}\geq 0$ 
с сохранением строгого возрастания очевидно. При этом
\begin{multline}\label{inth}
\int_0^{x}\frac{h(t)}{t^{\tt d-1}}\dd t
\leq {\sf s}_h\int_0^x \frac{\dd}{\dd t}\Bigl(\frac{h(t)}{t^{\tt d-2}}\Bigr)\dd t\\
={\sf s}_h \frac{h(x)}{x^{\tt d-2}}-{\sf s}_h \lim_{0<t\to 0}\frac{h(t)}{t^{\tt d-2}}
\leq {\sf s}_h \frac{h(x)}{x^{\tt d-2}} <+\infty
\quad \text{при всех $x\in [0,r]$}.
\end{multline}
Отсюда, в частности,  при выполнении  неравенства \eqref{ch} модуль непрерывности ${\sf h}_{\mu}$ 
удовлетворяет условию \eqref{emcont}, а кроме того, функция $h$ достигает значения ${\tt M}$
не правее, чем модуль непрерывности ${\sf h}_{\mu}$. Последнее по тождеству \eqref{h<M} предложения
\ref{proh} означает, что определено $h^{-1}({\tt M})\leq r$.  По теореме \ref{th1} 
для любой $\delta$-субгармонической функции $U\not\equiv \pm\infty$ выполнено   неравенство \eqref{URm}, где в правой части первые два сомножителя те же, что и в \eqref{URh}, а дополнительных преобразований в виде верхних оценок требует только последний сомножитель в скобках, заданный как сумма  
\begin{multline}\label{Mpr}
{\tt M}
\max\{1, r^{2-{\tt d}}\}+\widehat{\tt d}\int_0^{r}\frac{{\sf h}_{\mu}(t)}{t^{{\tt d}-1}}\dd t\\
=
{\tt M}\max\{1, r^{2-{\tt d}}\}+\widehat{\tt d}\int_0^{h^{-1}({\tt M})}\frac{{\sf h}_{\mu}(t)}{t^{{\tt d}-1}}\dd t
+\widehat{\tt d}\int_{h^{-1}({\tt M})}^{r}\frac{{\sf h}_{\mu}(t)}{t^{{\tt d}-1}}\dd t
\\
\overset{ \eqref{ch},\eqref{hRMl}}{\leq} {\tt M}\max\{1, r^{2-{\tt d}}\}+\widehat{\tt d}\int_0^{h^{-1}({\tt M})}\frac{h(t)}{t^{{\tt d}-1}}\dd t
+\widehat{\tt d}\int_{h^{-1}({\tt M})}^{r}\frac{{\tt M}}{t^{{\tt d}-1}}\dd t
\\
\overset{\eqref{inth}}{\leq}
{\tt M}\max\{1, r^{2-{\tt d}}\}+\widehat{\tt d}\,{\sf s}_h\frac{h\bigl(h^{-1}({\tt M})\bigr)}{\bigl(h^{-1}({\tt M})\bigr)^{{\tt d}-2}}
+\widehat{\tt d}\int_{h^{-1}({\tt M})}^{r}\frac{{\tt M}}{t^{{\tt d}-1}}\dd t 
\\
={\tt M}\biggl(\max\{1, r^{2-{\tt d}}\}+\frac{{\sf s}_h\widehat{\tt d}}{\bigl(h^{-1}({\tt M})\bigr)^{{\tt d}-2}}
+\Bigl(\Bbbk_{\tt d-2}(r)-\Bbbk_{\tt d-2}\bigl(h^{-1}({\tt M})\bigr)\Bigr)\biggr).
\end{multline}
При ${\tt d}=2$ правая часть здесь равна 
\begin{equation*}
{\tt M}\Bigl(1+{\sf s}_h+\ln\frac{r}{h^{-1}({\tt M})}\Bigr)={\tt M}\ln\frac{e^{1+{\sf s}_h}r}{h^{-1}({\tt M})}
\end{equation*}
и совпадает с фрагментом  правой части из \eqref{{URh}T}, содержащим ${\tt M}$. 

При ${\tt d}>2$ правая часть \eqref{Mpr} равна 
\begin{multline*}
{\tt M}\biggl(\max\{1, r^{2-{\tt d}}\}+\frac{{\sf s}_h({\tt d}-2)}{\bigl(h^{-1}({\tt M})\bigr)^{{\tt d}-2}}
+\biggl(\frac{1}{\bigl(h^{-1}({\tt M})\bigr)^{\tt d-2}} -\frac{1}{r^{\tt d-2}}\biggr)\Biggr)
\\
={\tt M}\biggl((1- r^{2-{\tt d}})^++\frac{{1+\sf s}_h({\tt d}-2)}{\bigl(h^{-1}({\tt M})\bigr)^{{\tt d}-2}}
\biggr)
\leq  {\tt M}\biggl(1+\frac{1+({\tt d}-2){\sf s}_h}{\bigl(h^{-1}({\tt M})\bigr)^{\tt d-2}}\biggr),
\end{multline*}
где правая часть совпадает с фрагментом  правой части  \eqref{{URh}2}, содержащим  ${\tt M}$. 
\end{proof}

\begin{remark}\label{rem1} 
Условие \eqref{sh}  теоремы \ref{th2} можно записать и в виде 
\begin{equation*}
\frac{1}{{\sf s}_h}:=\inf_{-\infty<x<\ln r} \bigl(\ln h(e^x)\bigr)'_x-({\tt d-2})>0,
\end{equation*}
а  условия  непрерывности функции $h$ и  её дифференцируемости  на $(0,r)$
 можно заменить на одно условие выпуклости функции 
 $h$ относительно $\ln$. Тогда $h$ непрерывна,  существует правая производная функции $h$   на $(0, r)$, а  производную $h'$ в  \eqref{sh} допустимо заменить на правую производную от $h$.
\end{remark}

\section{Обхват и мера Хаусдорфа в интегральных неравенствах}\label{mch}
\begin{definition}[{(\cite{Carleson}, \cite{Federer}, \cite{Rodgers}, \cite{HedbergAdams}, \cite{EG}, \cite{Eid07}, \cite{VolEid13})}]\label{defH}
Для функции   $h\colon \RR^+\to \RR^+$  и величины  $t\in \overline \RR^+\setminus 0$ 
функцию множеств 
\begin{equation}\label{mr}
{\mathfrak m}_h^{\text{\tiny $t$}}\colon S\underset{S\subset \RR^{\tt d}}{\longmapsto}  \inf \Biggl\{\sum_{j\in N} h(r_j)\biggm| N\subset \NN,\,  S\subset \bigcup_{j\in N} 
\overline B_{x_j}(r_j), \, x_j\in \RR^{\tt d}, \, r_j\underset{j\in N}{<} t\Biggr\} 
\end{equation}
со значениями в  $\overline \RR^+$ называем {\it $h$-обхватом  Хаусдорфа радиуса\/} 
$t$. Для каждого $S\subset \RR^{\tt d}$ значения    ${\mathfrak m}_h^{\text{\tiny $t$}}(S)$ убывают по $t$ и существует предел   
\begin{equation}\label{hH}
{\mathfrak m}_h^{\text{\tiny $0$}}(S):=\lim_{0<t\to 0} {\mathfrak m}_h^{\text{\tiny $t$}}(S)
\geq {\mathfrak m}_h^{\text{\tiny $t$}}(S)\geq {\mathfrak m}_h^{\text{\tiny $\infty$}}(S)
 \quad \text{\it для любого  $S\subset \RR^{\tt d}$}, 
\end{equation}
При $h(0)=0$ все обхваты ${\mathfrak m}_h^{\text{\tiny $t$}}$ --- внешние меры,  а ${\mathfrak m}_h^{\text{\tiny $0$}}$ определяет  {\it $h$-меру Хаусдорфа\/} ${\mathfrak m}_h^{\text{\tiny $0$}}$, являющуюся регулярной мерой Бореля.
Для  степенной   функции $h_p$ степени $p\in \RR^+$   с нормирующим множителем вида
\begin{equation}\label{hd}
h_p\colon x\underset{t\in \RR^+}{\longmapsto} 
c_px^p, \quad\text{где } 
c_p:=\dfrac{\pi^{p/2}}{\Gamma(p/2+1)}, \quad \Gamma
\text{ --- гамма-функция},
\end{equation}
 $h_p$-обхват радиуса  $t$ и  $h_p$-меру Хаусдорфа 
называем соответственно  {\it $p$-мер\-н\-ы\-ми  обхватом радиуса  $t$\/} и\/ {\it мерой Хаусдорфа\/}, которые обозначаем как
\begin{equation}\label{p-m}
p\text{\tiny-}{\mathfrak m}^{\text{\tiny $t$}}:={\mathfrak m}_{h_p}^{\text{\tiny $t$}}, \quad  
p\text{\tiny-}{\mathfrak m}^{\text{\tiny $0$}}:={\mathfrak m}_{h_p}^{\text{\tiny $0$}}. 
\end{equation}
\end{definition}

Здесь и далее классические и широко известные свойства обхватов и мер Хаусдорфа из основных источников, указанных в  начале определения  \ref{defH}, часто  используются без явно прописанных  конкретных  ссылок. 

\begin{example}\label{ex:1} Линейная мера Лебега $\uplambda_{\RR}$ на $\RR$ и плоская мера Лебега $\uplambda_{\CC}$ на $\CC$, использованные в \S\S~1--2,   совпадают соотвественно  с  одномерной   мерой Хаусдорфа $1\text{\tiny-}{\mathfrak m}^{\text{\tiny $0$}}$ на $\RR$ и с двумерной мерой Хаусдорфа $2\text{\tiny-}{\mathfrak m}^{\text{\tiny $0$}}$ на $\CC$,  а число элементов множества $S$   --- это $0$-мерная  мера Хаусдорфа этого множества $0\text{\tiny-}{\mathfrak m}^{\text{\tiny $0$}}(S)$, но для большей корректности полезно учесть \cite{Tuzhilin}--\cite{Tuzhilinr}. 
Кроме того, $\tt d$-мерная {\it пространственная мера Лебега\/} $\uplambda_{\RR^{\tt d}}$ на $\RR^{\tt d}$  совпадает с  
${\tt d}$-мерной  мерой Хаусдорфа ${\tt d}\text{\tiny-}{\mathfrak m}^{\text{\tiny $0$}}(S)$. 
Если  $p>{\tt d}$, то   $p$-мерная  мера Хаусдорфа  $p\text{\tiny-}{\mathfrak m}^{\text{\tiny $0$}}$ в $\RR^{\tt d}$ нулевая.
\end{example}

Неоднократно будет использована следующая фундаментальная 

\begin{thF}[{\rm (\cite[теорема II.1]{Carleson}, \cite[теорема 5.1.12]{HedbergAdams})}] 

\begin{enumerate}[{\rm I.}]
\item\label{IF} Если   $h\colon \RR^+\to \RR^+$ --- некоторая функция,
 а $\mu$ --- мера Бореля на $\RR^{\tt d}$
с модулем непрерывности 
\begin{equation}\label{hX}
{\sf h}_{\mu}(t)\overset{\eqref{hmuR}}{:=}\sup_{x\in \RR^{\tt d}}\mu \bigl(\overline B_x(t)\bigr)\leq h(t) \quad\text{при  всех  $t\in \RR^+$},
\end{equation} 
то
\begin{equation}\label{zv}
\mu(S)\leq {\mathfrak m}_h^{\text{\tiny $\infty$}}(S) \quad\text{для любого $\mu$-измеримого  $S\subset \RR^{\tt d}$}.
\end{equation}

\item\label{IIF} 
Существует такое  число $A>0$,
что для каждой возрастающей функции $h\colon \RR^+\to \RR^+$  с $h(0)=0$ и  для любого 
компакта $E\subset \RR^{\tt d}$ найдётся мера Радона $\mu$ на $E$, для которой 
выполнено \eqref{hX} и, как следствие, \eqref{zv},  а также одновременно неравенство  
противоположной направленности 
\begin{equation}\label{muhr}
A\mu(E) \geq {\mathfrak m}_h^{\text{\tiny $\infty$}}(E).
\end{equation} 
\end{enumerate}
\end{thF}
\begin{remark}
Обе части \ref{IF} и \ref{IIF} теоремы Фростмана в известных нам  формулировках даются с едиными посылками и, как следствие, с перегрузкой условий на функцию $h$ в части \ref{IF}, которая  
для $N$, $x_j$ и  $r_j$  из требования в фигурных скобках из  \eqref{mr}
получается применением к крайним частям неравенств
\begin{equation*}
\mu(S)\leq \sum_{j\in N}\mu\bigl(B_{x_j}(r_j)\bigr)\overset{\eqref{hX}}{\leq} \sum_{j\in N}h(r_j) 
\end{equation*}
операции $\inf$ по всем  таким $N$, $x_j$ и  $r_j$.

 Содержательность значительно более глубокой части \ref{IIF} теоремы Фростмана для настоящей статьи 
уже  в том, что для теорем \ref{th1} и \ref{th2} она обеспечивает существование ненулевых мер $\mu$, удовлетворяющих условиям этих теорем. 
\end{remark}

\begin{theorem}\label{th3} Пусть мера Бореля $\mu$ полной меры\/ $\tt M$ с модулем непрерывности ${\sf h}_{\mu}$ сосредоточена на $\mu$-измеримом множестве $S\subset \overline B(r)$.   Тогда 
\begin{equation}\label{smu}
{\tt M}= {\mathfrak m}_{{\sf h}_{\mu}}^{\text{\tiny $\infty$}}(S)
={\mathfrak m}_{{\sf h}_{\mu}}^{\text{\tiny $t$}}(S)
\quad\text{при любом радиусе обхвата $t\geq r$},
\end{equation} 
а для любой функции $h\colon [0,r]\to \RR^+$ при $h\geq {\sf h}_{\mu}$ на $[0,r]$ и  продолжении $h$ на луч $(r,+\infty)$ 
значением $h(r)$ имеют место неравенства
\begin{equation}\label{smuh}
{\tt M}\leq {\mathfrak m}_{h}^{\text{\tiny $\infty$}}(S)
\leq {\mathfrak m}_{h}^{\text{\tiny $t$}}(S)
\quad\text{при любом радиусе обхвата $t\in \overline \RR^+$}.
\end{equation} 
В частности, 
\begin{enumerate}[{\rm (i)}]
\item\label{ih} в итоговом   неравенстве \eqref{URm} теоремы\/ {\rm \ref{th1}}
 можно  заменить в  правой части полную меру\/ ${\tt M}$ на ${\sf h}_{\mu}$-обхват\/  ${\mathfrak m}_{{\sf h}_{\mu}}^{\text{\tiny $t$}}(S)$  любого радиуса обхвата $t\in \overline \RR^+$;
\item\label{iih} в  заключительных неравенствах \eqref{{URh}T} и \eqref{{URh}2} теоремы\/ {\rm \ref{th2}}  можно заменить парные вхождения в правых частях  полной меры\/ ${\tt M}$  одновременно на $h$-обхват  ${\mathfrak m}_h^{\text{\tiny $t$}}(S)$  множества $S$ любого радиуса обхвата $t\in \overline \RR^+$.
\end{enumerate}
\end{theorem}
\begin{proof}
По части \ref{IF} теоремы Фростмана при $h:={\sf h}_{\mu}$ имеем
\begin{equation}\label{F1}
{\tt M}=\mu(\RR^{\tt d})=\mu(S)\leq 
{\mathfrak m}_{{\sf h}_{\mu}}^{\text{\tiny $\infty$}}(S)
\overset{\eqref{hH}}{\leq} {\mathfrak m}_{{\sf h}_{\mu}}^{\text{\tiny $t$}}(S),
\end{equation}
откуда сразу получаем \eqref{smuh}. При радиусе обхвата $t\geq r$ шар $\overline B(r)$ включает в себя 
 $S$ и в то же время по определению \eqref{mr} имеем ${\mathfrak m}_{{\sf h}_{\mu}}^{\text{\tiny $t$}}(S)\leq {\sf h}_{\mu}(r)$, где по неравенству  \eqref{hRMl} предложения \ref{proh} правая часть не превышает $\tt M$. Это вместе с \eqref{F1} даёт  равенства  \eqref{smu}. По неравенствам  \eqref{smuh} утверждение \eqref{ih} очевидно. 

Для доказательства  утверждения  \eqref{iih} потребуется 

\begin{lemma}\label{lem2} Пусть  функция $h$ такая же, как в теореме\/ {\rm \ref{th2}},  с числом  ${\sf s}_h>0$ 
из \eqref{sh}, а значит строго возрастающая на $[0,r]$, и  пусть $r\leq B\in \RR^+$.  Тогда
\begin{subequations}\label{h-1}
\begin{align}
x&\underset{x\in [0, h(r)]}{\longmapsto}
\frac{x}{\bigl(h^{-1}(x)\bigr)^{\tt d-2}} 
 \quad \text{при ${\tt d}>2$},
\tag{\ref{h-1}{\tt d}}\label{{h-1}d}
\\
x&\underset{x\in [0, h(r)]}{\longmapsto} x\ln \frac{Be^{{\sf s}_h}}{h^{-1}(x)} \quad \text{при ${\tt d=2}$}
\tag{\ref{h-1}{$\CC$}}\label{{h-1}C}
\end{align}
\end{subequations}
--- возрастающие  функции на отрезке  $\bigl[0, h(r)\bigr]$.
\end{lemma}
\begin{proof}[леммы \ref{lem2}] Произведём замену $y:=h^{-1}(x)\in [0,r]$ и перейдём от функций 
\eqref{h-1} к функциям  
\begin{subequations}\label{h-1-}
\begin{align}
y&\overset{\eqref{{h-1}d}}{\underset{y\in [0,r]}{\longmapsto}}
\frac{h(y)}{y^{\tt d-2}}  \quad \text{при ${\tt d}>2$},
\tag{\ref{h-1-}{\tt d}}\label{{h-1-}d}
\\
y&\overset{\eqref{{h-1}C}}{\underset{y\in [0,r]}{\longmapsto}} h(y)\ln \frac{Be^{{\sf s}_h}}{y} \quad \text{при ${\tt d=2}$},
\tag{\ref{h-1-}{$\CC$}}\label{{h-1-}C}
\end{align}
\end{subequations}
Ввиду строгого возрастания непрерывной функции $h$  на $[0,r]$ достаточно показать, что возрастают функции  \eqref{h-1-}. Строгое возрастание функции \eqref{{h-1-}d} было обосновано  в \eqref{htd} и ниже при доказательстве теоремы \ref{th2}.  Для функции \eqref{{h-1-}C} её дифференцирование  на  $(0,r)$ даёт 
\begin{equation*}
\frac{\dd}{\dd y}h(y)\ln \frac{Be^{{\sf s}_h}}{y}=h'(y)\ln \frac{Be^{{\sf s}_h}}{y}-\frac{h(y)}{y}\overset{\eqref{sh}}{\geq} 
h'(y)\ln \frac{Be^{{\sf s}_h}}{y}-{\sf s}_hh'(y)=h'(y)\ln \frac{B}{y}\geq 0
\end{equation*} 
на $(0,r)$ при $B\geq r$, откуда следует возрастание функции \eqref{{h-1-}C} на   $(0,r)$, а в силу непрерывности и на отрезке $[0,r]$. 
\end{proof}

Теперь по лемме \ref{lem2}   в силу возрастания функции \eqref{{h-1}d}
в  заключительном  неравенстве \eqref{{URh}2} теоремы\/ {\rm \ref{th2}}
 можно заменить\/ ${\tt M}$  на $h$-обхват  ${\mathfrak m}_h^{\text{\tiny $t$}}(S)
\overset{\eqref{smuh}}{\geq} {\tt M}$  множества $S$ любого радиуса обхвата $t\in \overline \RR^+$.

Для   заключительного неравенства \eqref{{URh}T}  теоремы\/ {\rm \ref{th2}} снова по 
лемме \ref{lem2}   в силу возрастания функции \eqref{{h-1}C} с $B:=er\geq r$ 
можно заменить\/ ${\tt M}$  на $h$-обхват  ${\mathfrak m}_h^{\text{\tiny $t$}}(S)\overset{\eqref{smuh}}{\geq} {\tt M}$  множества $S$ любого радиуса обхвата $t\in \overline \RR^+$.
\end{proof}

Допускаемая теоремой \ref{th3}\eqref{ih} замена  в неравенстве  \eqref{URm} теоремы {\rm \ref{th1}}  
полной меры ${\tt M}$ на  $h$-обхват  Хаусдорфа ${\mathfrak m}_{h}^{\text{\tiny $\infty$}}(S)$ при радиусе обхвата  $t\geq r$
согласно равенствам \eqref{smu} не ослабляет это неравенство. Но и для заключительных неравенств \eqref{{URh}2} и  \eqref{{URh}T} теоремы\/ {\rm \ref{th2}} при любой функции $h$ возможны ситуации, когда 
 замена полной меры ${\tt M}$ на  $h$-обхват  Хаусдорфа ${\mathfrak m}_{h}^{\text{\tiny $\infty$}}(S)$ радиуса $+\infty$
ослабляет  эти оценки  разве что  на абсолютную постоянную-множитель, что отражает 

\begin{theorem}\label{th4}
Существует такая абсолютная постоянная $A\geq 1$, что для любого  $r\in \RR^+\setminus 0$, для всякого  
 компакта $S\subset \overline B(r)$ и  для каждой функции $h\colon [0,r]\to \RR^+$, удовлетворяющей всем условиям теоремы\/ {\rm \ref{th2}} с постоянной ${\sf s}_h\overset{\eqref{sh}}{>}0$,
  найдётся  такая мера Бореля  $\mu$ на $\overline B(r)$ полной меры ${\tt M}>0$,  с носителем 
$\supp \mu \subset S$ и с модулем непрерывности, 
удовлетворяющим  \eqref{ch}, что одновременно  с неравенствами \eqref{URh} как с\/ $\tt M$, так и с\/ ${\mathfrak m}_{h}^{\text{\tiny $\infty$}}(S)$ вместо $\tt M$
 для произвольной $\delta$-субгармонической функции $U\not\equiv\pm\infty$ на шаре $\overline B(R)$ радиуса $R>r$  выполнены  и неравенства   с множителем $A$ перед\/ ${\tt M}$  вида
\begin{subequations}\label{inFp-}
\begin{align}
A\, {\tt M}\ln \frac{e^{1+{\sf s}_h}r}{h^{-1}({\tt M})}&\geq   {{\mathfrak m}_{h}^{\text{\tiny $\infty$}}(S)}\ln \frac{e^{1+{\sf s}_h}r}{h^{-1}({\mathfrak m}_{h}^{\text{\tiny $\infty$}}(S))}
 \quad \text{при ${\tt d=2}$}, 
\tag{\ref{inFp-}{$\CC$}}\label{{h-1}CA}
\\
A\, {\tt M}\biggl(1+\frac{1+({\tt d}-2){\sf s}_h}{\bigl(h^{-1}({\tt M})\bigr)^{\tt d-2}}
\biggr)
&\geq   {{\mathfrak m}_{h}^{\text{\tiny $\infty$}}(S)}
\biggl(1+\frac{1+({\tt d}-2){\sf s}_h}{h^{-1}({\mathfrak m}_{h}^{\text{\tiny $\infty$}}(S))}\biggr)
 \quad \text{при ${\tt d}>2$}.
\tag{\ref{inFp-}{\tt d}}\label{{h-1}dA}
\end{align}
\end{subequations}
\end{theorem}
\begin{proof}
Пусть  $E:=S\subset \overline{B}(r)$ и  $h$ --- функция из условия теоремы \ref{th2} с постоянным  продолжением значением $h(r)$ на луч $(r,+\infty)$.  По части \ref{IIF} теоремы Фростмана выберем меру Радона $\mu\neq 0$ с всеми прописанными в этой части  \ref{IIF} свойствами. Тогда выполнены условия теорем \ref{th2} и \ref{th3} 
с их заключениями соответственно \eqref{URh} и \eqref{iih} и в то же время  из  \eqref{smuh} и \eqref{muhr}  следует 
$$ 
{\tt M}=\mu(S)\overset{\eqref{smuh}}{\leq} {\mathfrak m}_h^{\text{\tiny $\infty$}}(S) \overset{\eqref{muhr}}{\leq} A\mu(S)=A\,{\tt M}.
$$
Отсюда, учитывая возрастание функции  $h^{-1}$ в знаменателях левых частей 
\eqref{{h-1}dA}и \eqref{{h-1}CA},  получаем  оба неравенства из \eqref{inFp-}.
\end{proof}

\section{Частные случаи  неравенств для интегралов от разностей субгармонических функций}\label{S7}

\subsection{Случай $p$-мерных обхватов и мер Хаусдорфа}\label{p7_1}

\begin{theorem}\label{th5} Пусть  $0<r\leq t\in \overline \RR^+$, $p\in ({\tt d}-2,{\tt d}]$, $b\in \RR^+$. Для любой  меры Бореля  $\mu$ на  $\overline B(r)$ с носителем  $\supp \mu\subset  S \subset \overline B(r)$ и модулем непрерывности
\begin{equation}\label{{chrh}hd}
{\sf h}_{\mu}(x)\overset{\eqref{sh}}{\leq} bx^p\quad \text{при  всех $x\in [0,r]$}
\end{equation}
каждая    $\delta$-суб\-г\-а\-р\-м\-о\-н\-и\-ч\-е\-с\-к\-ая   функция  $U\not\equiv \pm \infty$ на шаре  $\overline B(R)$
 радиуса $R>r$  $\mu$-суммируема и 
\begin{subequations}\label{Uhd}
\begin{align}
\int_{\overline D(r)}
 U^+\dd \mu &\leq \frac{b}{p}\frac{R+r}{R-r} {\boldsymbol  T}_U(r,R)\, p\text{\tiny-}{\mathfrak m}^{\text{\tiny $t$}}(S)
\ln\frac{\pi e^{p+1}r^p}{p\text{\tiny-}{\mathfrak m}^{\text{\tiny $t$}}(S)}\quad\text{при ${\tt d=2}$ для\/ $\CC$},
\tag{\ref{Uhd}$\CC$}\label{{Uhd}C}\\
\int_{\overline B(r)} U^+\dd \mu &\leq b
{\tt d}^{{\tt d}}A_{\tt d}(r,R) {\boldsymbol  T}_U(r,R)\notag\\
&\times p\text{\tiny-}{\mathfrak m}^{\text{\tiny $t$}}(S)
\Biggl(1+\frac{1}{\bigl(p-({\tt d-2})\bigr)\bigl(p\text{\tiny-}{\mathfrak m}^{\text{\tiny $t$}}(S)\bigr)^{\frac{\tt d-2}{p}}}\Biggr)\quad\text{при  ${\tt d}>2$,}
\tag{\ref{Uhd}{\tt d}}\label{{Uhd}d}
\end{align}
\end{subequations}  
 где  $r$ в ${\boldsymbol   T}_U(r,R)$  можно заменить на любое число $r'\in [0,r]$. 
\end{theorem}
\begin{proof}
Положим  $h(x):=bx^p$ при всех  $x\in [0,r]$, откуда  
\begin{equation}\label{h-1x}
\frac1{{\sf s}_h}\overset{\eqref{sh}}{=}p-({\tt d-2})>0, \quad h^{-1}(y)=\Bigl(\frac{y}{b}\Bigr)^{1/p}, \quad h\overset{\eqref{hd}}{=}\frac{b}{c_p}h_p,
\end{equation}
По условию \eqref{{chrh}hd} выполнено и условие \eqref{ch} теоремы \ref{th2}, из применения которой вместе с теоремой \ref{th3}\eqref{iih} следует 
\begin{equation*}
\begin{split}
\int U^+\dd \mu \leq&  
5\frac{R+r}{R-r} {\boldsymbol  T}_U(r,R) \, {\mathfrak m}_h^{\text{\tiny $t$}}(S)\ln\frac{e^{1+\frac{1}{p}}r}{h^{-1}(
{\mathfrak m}_h^{\text{\tiny $t$}}(S))}
\quad\text{при ${\tt d}=2$,}
\\
\int U^+\dd \mu \leq&  
A_{\tt d}(r,R) {\boldsymbol  T}_U(r,R)\, {\mathfrak m}_h^{\text{\tiny $t$}}(S)
\biggl(1+\frac{1+\frac{p}{p-({\tt d-2})}}{\bigl(h^{-1}({{\mathfrak m}_h^{\text{\tiny $t$}}(S)})\bigr)^{\tt d-2}}\biggr)
\quad\text{при  ${\tt d}>2$,}
\end{split}
\end{equation*}
что согласно  равенству   
\begin{equation*}
h^{-1}({{\mathfrak m}_h^{\text{\tiny $t$}}(S)})\overset{\eqref{h-1x}}{=}
\Bigl(\frac{{\mathfrak m}_h^{\text{\tiny $t$}}(S)}{b}\Bigr)^{1/p}
\end{equation*}
можно переписать как 
\begin{subequations}\label{URhp+}
\begin{align}
\int U^+\dd \mu \leq&  5\frac{R+r}{R-r} {\boldsymbol  T}_U(r,R) \, {\mathfrak m}_h^{\text{\tiny $t$}}(S)\frac{1}{p}
\ln\frac{be^{p+1}r^p}{{\mathfrak m}_h^{\text{\tiny $t$}}(S)}
\quad\text{при ${\tt d}=2$, т.е. в\/ $\CC$,}
\tag{\ref{URhp+}$\CC$}\label{{URh}Tp+}
\\
\int U^+\dd \mu \leq&  A_{\tt d}(r,R) {\boldsymbol  T}_U(r,R)\, {\mathfrak m}_h^{\text{\tiny $t$}}(S)
\biggl(1+\frac{b^{\frac{\tt d-2}{p}}\bigl(2p-({\tt d-2})\bigr)}{\bigl(p-({\tt d-2})\bigr)\bigl({{\mathfrak m}_h^{\text{\tiny $t$}}(S)}\bigr)^{\frac{\tt d-2}{p}}}\biggr)
\tag{\ref{URhp+}{\tt d}}\label{{URh}2p+}
\end{align}
\end{subequations}
при  ${\tt d}>2$. При этом по определению \ref{defH} $h$-обхвата Хаусдорфа в \eqref{mr} и $p$-мерного обхвата 
Хаусдорфа в \eqref{p-m} по последнему равенству в \eqref{h-1x} имеем равенства
$$
{\mathfrak m}_h^{\text{\tiny $t$}}\overset{\eqref{h-1x}}{=}\frac{b}{c_p}
{\mathfrak m}_{h_p}^{\text{\tiny $t$}}\overset{\eqref{p-m}}{=}
\frac{b}{c_p}p\text{\tiny-}{\mathfrak m}^{\text{\tiny $t$}},
$$
и подстановка правой части в \eqref{{URh}Tp+} и в \eqref{{URh}2p+} даёт соответственно 
 \begin{subequations}\label{URhp+1}
\begin{align}
\int U^+\dd \mu \leq&  
\frac{5b}{pc_p}\frac{R+r}{R-r} {\boldsymbol  T}_U(r,R)\, p\text{\tiny-}{\mathfrak m}^{\text{\tiny $t$}}(S)
\ln\frac{c_pe^{1+p}r^p}{p\text{\tiny-}{\mathfrak m}^{\text{\tiny $t$}}(S)}
\quad\text{при ${\tt d}=2$,}
\tag{\ref{URhp+1}$\CC$}\label{{URh}Tp+1}
\\
\int U^+\dd \mu \leq& bA_{\tt d}(r,R) {\boldsymbol  T}_U(r,R)p\text{\tiny-}{\mathfrak m}^{\text{\tiny $t$}}(S)
\, \biggl(\frac{1}{c_p}+\frac{c_p^{\frac{{\tt d-2}}{p}-1}({\tt d}+2)}{\bigl(p-({\tt d-2})\bigr)
\bigl(p\text{\tiny-}{\mathfrak m}^{\text{\tiny $t$}}(S)\bigr)^{\frac{\tt d-2}{p}}}\biggr)
\tag{\ref{URhp+1}{\tt d}}\label{{URh}2p+1}
\end{align}
\end{subequations}
при  ${\tt d}>2$ ввиду $p\leq {\tt d}$. При $\tt d=2$ из определения \eqref{hd}  имеем  оценку сверху  
\begin{equation*}
\pi \geq c_p\overset{\eqref{hd}}{:=}\dfrac{\pi^{p/2}}{\Gamma(p/2+1)}\geq 1
\quad\text{при $p\in (0,2]$,}
\end{equation*}
что согласно \eqref{{URh}Tp+1} влечёт за собой \eqref{{Uhd}C}.

При ${\tt d}\geq 3$ и  $p\in ({\tt d}-2,{\tt d}]$  ввиду $\frac{{\tt d-2}}{p}-1<0$ получаем 
$$
c_p^{\frac{{\tt d-2}}{p}-1}= \biggl(\frac{\Gamma(p/2+1)}{\pi^{p/2}}\biggr)^{1-\frac{{\tt d-2}}{p}}
\leq \bigl(\Gamma(p/2+1)\bigr)^{1-\frac{{\tt d-2}}{p}}
\leq \bigl(\Gamma({\tt d}/2+1)\bigr)^{2/{\tt d}}\leq \frac{{\tt d}}{2},
$$
откуда для последней скобки в \eqref{{URh}2p+1}
\begin{multline*}
\biggl(\frac{1}{c_p}+\frac{c_p^{\frac{{\tt d-2}}{p}-1}({\tt d+2})}{\bigl(p-({\tt d-2})\bigr)
\bigl(p\text{\tiny-}{\mathfrak m}^{\text{\tiny $t$}}(S)\bigr)^{\frac{\tt d-2}{p}}}\biggr)
\\
\leq 
({\tt d}/2)^{{\tt d}/2}+\frac{{\tt d}({\tt d+2})/2}{\bigl(p-({\tt d-2})\bigr)
\bigl(p\text{\tiny-}{\mathfrak m}^{\text{\tiny $t$}}(S)\bigr)^{\frac{\tt d-2}{p}}}
\\
\leq 
{\tt d}^{{\tt d}}\biggl(1+\frac{1}
{\bigl(p-({\tt d-2})\bigr)
\bigl(p\text{\tiny-}{\mathfrak m}^{\text{\tiny $t$}}(S)\bigr)^{\frac{\tt d-2}{p}}}
\biggr),
\end{multline*}
что согласно \eqref{{URh}2p+1} влечёт за собой \eqref{{Uhd}d}.
\end{proof}

\subsection{Функции на комплексной плоскости и пространстве $\RR^{\tt d}$}\label{37all}

В оценках сверху в правых частях  один из первых сомножителей  $(R+r)/(R-r)$ в случае комплексной плоскости, как и явно выписываемый сомножитель  $A_{\tt d}$ из \eqref{{UR}A}
для  $\RR^{\tt d}$ с ${\tt d}>2$ позволяют в  явном виде учитывать  близость $R>r$ к $r$. Не менее важным может оказаться случай  существенной  удалённости $R$ от $r$, когда функции и меры рассматриваются на всём $\RR^{\tt d}$, а  характеристика Неванлинны достаточно медленно растёт. 
\begin{corollary}\label{corCRd}
Пусть $\mu$ ---  мера Радона  на  $\RR^{\tt d}$, $U\not\equiv \pm \infty$ ---  $\delta$-субгармон\-и\-ч\-е\-с\-к\-ая фу\-н\-к\-ция на всём  $\RR^{\tt d}$, а функция $s\colon \RR^+\to \RR^+\setminus 0$ произвольная.  Тогда 
\begin{enumerate}[{\rm I.}]
\item\label{allI} 
Если для каждого $R\in \RR^+$ %%\/ {\rm (ср. с \eqref{k0CR})}
$$
\sup_{y\in \overline B(R)}{\sf N}_y^{\mu}(r_0)<+\infty \quad\text{при некотором $r_0>0$},
$$
то функция $U$ локально суммируема по мере $\mu$   и 
\begin{multline*}
\int_{\overline B(r)} U^+\dd \mu \leq 
5{\tt d} \Bigl(1+\frac{2r}{s(r)}\Bigr)^{\tt d-1} 
\bigl(1+s(r)\bigr)^{\tt d-2}{\boldsymbol  T}_U\bigl(r,r+s(r)\bigr)\\
\times \biggl(\mu^{\rad}(r)\max\{1, r^{2-\tt d}\} +\sup_{y\in  \RR^{\tt d}}{\sf N}_y^{\mu}(r)\biggr)\quad\text{при любом $r\in \RR^+$}.
\end{multline*}
\item\label{allII} Если  $h\colon \RR^+\to \RR^+$ --- непрерывная функция с $h(0)=0$, дифференцируемая на $\RR^+\setminus 0$, и  при любом $r\in \RR^+$, т.е. при $r:=+\infty$, выполнены условия 
\eqref{sh} и \eqref{ch}, то функция $U$ локально $\mu$-суммируема и  
\begin{subequations}
\begin{align}
\int_{\overline D(r)} U^+\dd \mu \leq&  
5\biggl(1+\frac{2r}{s(r)}\biggr) {\boldsymbol  T}_U\bigl(r,r+s(r)\bigr) \, \mu^{\rad}(r)\ln\frac{e^{1+{\sf s}_h}r}{h^{-1}( \mu^{\rad}(r))}
\quad\text{на\/ $\CC$,}\notag
\\
\int_{\overline B(r)} U^+\dd \mu \leq&  
5{\tt d} \Bigl(1+\frac{2r}{s(r)}\Bigr)^{\tt d-1} 
\bigl(1+s(r)\bigr)^{\tt d-2}{\boldsymbol  T}_U\bigl(r,r+s(r)\bigr)\notag\\
&\times  \mu^{\rad}(r)\biggl(1+\frac{1+({\tt d}-2){\sf s}_h}{\bigl(h^{-1}(\mu^{\rad}(r))\bigr)^{\tt d-2}}
\biggr)
\quad\text{при  ${\tt d}>2$,}
\notag
\end{align}
\end{subequations}
где в правых частях обоих неравенств парные вхождения  $\mu^{\rad}(r)$  можно заменить 
одновременно на $h$-обхват Хаусдорфа ${\mathfrak m}_h^{\text{\tiny $\infty$}}\bigl(\overline B(r)\cap \supp \mu\bigr)$  бесконечного диаметра или на $h$-меру Хаусдорфа 
${\mathfrak m}_h^{\text{\tiny $0$}}\bigl(\overline B(r)\cap \supp \mu\bigr)$
части носителя меры Радона $\mu$, попавшей соответственно в круг $\overline D(r)\subset \CC$ или в шар $\overline B(r)\subset \RR^{\tt d}$ при   ${\tt d}>2$.

Всюду первый аргумент $r$ в ${\boldsymbol  T}_U\bigl(r,r+s(r)\bigr)$ из правых  частей приведённых неравенств можно заменить на любое число   $r'\in [0,r]$.
\end{enumerate} 
\end{corollary}

\begin{proof} Утверждение \ref{allI} --- это переписанное при $R:=r+s(r)>r$ неравенство 
\cite[теорма-критерий, (2.2T)]{Kha21N2}  для сужений $\mu\bigm|_{\overline B(r)}$ меры $\mu$ на $\overline B(r)$ с некоторыми огрублениями-упрощениями  для 
\begin{multline}\label{Adrog}
A_{\tt d}\bigl(r,r+s(r)\bigr)\overset{\eqref{{UR}A}}{:=}
5\max\bigl\{1, {\tt d}-2\bigr\}\Bigl(\frac{\bigl(r+s(r)\bigr)+r}{s(r)}\Bigr)^{\tt d-1}\max\Bigl\{1, \bigl(s(r)\bigr)^{\tt d-2}\Bigr\} 
\\
\leq 5({\tt d}-1)
 \Bigl(1+\frac{2r}{s(r)}\Bigr)^{\tt d-1} 
\max\Bigl\{1, \bigl(s(r)\bigr)^{\tt d-2}\Bigr\},
\end{multline}
где последний максимум допустимо заменить на б\'ольшую сумму  $\bigl(1+s(r)\bigr)^{\tt d-2}$.

Утверждение \ref{allII}  получается применением теоремы \ref{th2} при каждом $r\in \RR^+$ к сужениям меры Радона $\mu$ соответственно на круги $\overline D(r)$ или шары $\overline B(r)$ с  неравенствами   \eqref{URh} с учётом \eqref{Adrog}, а также использованием  
теоремы \ref{th3}\eqref{iih} в заключительной части.
\end{proof}

\subsection{Функции на круге $\DD:=D(1)$ и шаре  $\BB:=B(1)\subset \RR^{\tt d}$}\label{37allDB}
Предшествующие результаты, приведённые  во введении, не приспособлены для применения  к мероморфных функциям и к разностям субгармонических  функций на единичном круге. Всюду в этом п.~\ref{37allDB} ниже 
$\mu$ ---  мера Бореля, сосредоточенная в  $\BB\subset \RR^{\tt d}$ и конечная  на  $r\overline \BB=\overline B(r)$ при каждом $r\in [0,1)$,  $U\not\equiv \pm \infty$ --- $\delta$-субгармоническая функция  на $\BB$,  а функция
$s\colon [0,1)\to \RR^+$такова, что
\begin{equation}\label{sr1}
 0<s(r)<1-r\quad\text{при всех $r\in [0,1)$,}
\end{equation}

\begin{corollary}\label{corDB}
 Пусть ${\tt d}\geq 2$ и $\mu_r:=\mu\bigm|_{r\BB}$ --- сужение меры $\mu$ на $r\BB$.  
Если для каждого $r\in [0,1)$ %%\/ {\rm (ср. с \eqref{k0CRr})}
\begin{equation}\label{Nmur1}
\sup_{y\in \supp \mu_r}
{\sf N}_y^{\mu_r}(r)<+\infty
 \quad\text{для некоторого $r_0>0$},
\end{equation}
 то при каждом $r\in [0,1)$ существует интеграл 
\begin{equation}\label{irD}
\int_{r\overline \BB} U^+\dd \mu \leq 
\frac{3^{2\tt d}}{\bigl(s(r)\bigr)^{\tt d-1}}
{\boldsymbol  T}_U\bigl(r,r+s(r)\bigr)\biggl(\frac{\mu^{\rad}(r)}{r^{\tt d-2}} +\sup_{y\in  \BB}{\sf N}_y^{\mu_r}(r)\biggr)
<+\infty.
\end{equation}
\end{corollary}
\begin{proof}
По условию  \eqref{Nmur1}  справедливо  \cite[теорема-критерий, ут\-в\-е\-р\-ж\-дение V]{Kha21N2} с неравенством 
\cite[(2.5)]{Kha21N2}   для сужения $\mu_r$ вместо $\mu$. Из импликации  
\cite[теорема-критерий, {V}$\Longrightarrow${II}]{Kha21N2}
 с учётом для $r<1$ неравенств 
\begin{multline}\label{Adrog+}
A_{\tt d}\bigl(r,r+s(r)\bigr)\overset{\eqref{Adrog}}{\leq} 5{(\tt d}-1)
 \Bigl(1+\frac{2r}{s(r)}\Bigr)^{\tt d-1} \max\Bigl\{1, \bigl(s(r)\bigr)^{\tt d-2}\Bigr\}\\
\overset{\eqref{sr1}}{\leq} \frac{5({\tt d}-1)3^{\tt d-1}}{\bigl(s(r)\bigr)^{\tt d-1}}
\leq \frac{1}{\bigl(s(r)\bigr)^{\tt d-1}}\begin{cases}
15\text{ при ${\tt d=2}$}\\
3^{2{\tt d}}\text{ при ${\tt d}\geq 2$}
\end{cases}\quad \text{при всех $r\in [0,1)$}
\end{multline}
по неравенству \cite[(2.2)]{Kha21N2} с $\mu_r$ вместо $\mu$ и $r+s(r)$ в роли $R$ получаем   \eqref{irD}.
\end{proof}

\begin{corollary}\label{corDBh}
  Если  $h\colon [0,1]\to \RR^+$ --- непрерывная функция с $h(0)=0$, дифференцируемая на $(0,1)$, и  при $r:=1$ выполнены \eqref{sh} и \eqref{ch}, то функция $U$  $\mu$-суммируема на каждом соответственно круге $r\overline \DD$ или шаре $r\overline \BB$ и  
\begin{align}
\int_{r\overline \DD} U^+\dd \mu &\leq  
\frac{15}{s(r)} {\boldsymbol  T}_U\bigl(r,r+s(r)\bigr) \, \mu^{\rad}(r)\ln\frac{e^{1+{\sf s}_h}r}{h^{-1}( \mu^{\rad}(r))}
\quad\text{на\/ $\DD$,}\notag
\\
\int_{r\overline \BB} U^+\dd \mu &\leq  
\frac{3^{2\tt d}}{\bigl(s(r)\bigr)^{\tt d-1}} 
{\boldsymbol  T}_U\bigl(r,r+s(r)\bigr)
  \mu^{\rad}(r)\biggl(1+\frac{1+({\tt d}-2){\sf s}_h}{\bigl(h^{-1}(\mu^{\rad}(r))\bigr)^{\tt d-2}}
\biggr)\text{ при\/  ${\tt d}>2$}
\notag
\end{align}
на\/ $\BB$, где справа в  обоих неравенств парные вхождения\/  $\mu^{\rad}(r)$  можно заменить одновременно на $h$-обхват Хаусдорфа\/ ${\mathfrak m}_h^{\text{\tiny $\infty$}}\bigl(\overline B(r)\cap \supp \mu\bigr)$  бесконечного диаметра или на $h$-меру Хаусдорфа ${\mathfrak m}_h^{\text{\tiny $0$}}\bigl(\overline B(r)\cap \supp \mu\bigr)$, а
 первый аргумент $r$ в ${\boldsymbol  T}_U\bigl(r,r+s(r)\bigr)$ справа  можно заменить на любое число   $r'\in [0,r]$.
\end{corollary}
\begin{proof} Получается сочетанием теорем \ref{th2} и \ref{th3}\eqref{iih} из \eqref{URh},
применяемого    при  $r\in (0,1)$ к сужениям  $\mu\bigm|_{r\BB}$  вместо $\mu$ и $r+s(r)$ в роли $R$, 
 и \eqref{Adrog+}.
\end{proof}
\subsection{Случай  ${\tt d}$-мерной пространственной меры Лебега в $\RR^{\tt d}$}\label{p7_2}

Рассмотрим в теореме \ref{th5}  в качестве меры Бореля  $\mu$ сужения соответственно плоской меры Лебега $\uplambda_{\CC}$ на $\uplambda_{\CC}$-измеримое $E\subset \overline D(r)$ или ${\tt d}$-мерной пространственной мерой Лебега $\uplambda_{\RR^{\tt d}}$ на $\uplambda_{\RR^{\tt d}}$-измеримое  $E\subset \overline B(r)$. Тогда в крайнем случае  теоремы \ref{th5}  при выборе $p:={\tt d}$, $t:=0$ и  $b$ в \eqref{{chrh}hd}, равным  соответственно площади $\pi$ единичного круга $D(1)\subset \CC$ при ${\tt d}=2$ или  объёму единичного шара  $B(1)\subset \RR^{\tt d}$ 
\begin{equation}\label{volB}
\frac{\pi^{{\tt d}/2}}{\Gamma({\tt d}/2+1)}\leq \frac{\pi^{{\tt 5}/2}}{\Gamma({\tt 5}/2+1)}=\frac{8}{15}\pi^2< 6
\quad\text{при всех ${\tt d}>2$,}
\end{equation} 
согласно совпадению меры Лебега $\uplambda_{\RR^{\tt d}}$ и  ${\tt d}$-мерной меры Хаусдорфа ${\tt d}\text{\tiny-}{\mathfrak m}^{\text{\tiny $0$}}$ в $\RR^{\tt d}$, отмеченному в  примере \ref{ex:1}, сразу получаем  
\begin{corollary}\label{cor2} При  $0<r\in \RR^+$ и $\uplambda_{\RR^{\tt d}}$-измеримом $E\subset \overline B(r)$ 
для каждой 
 $\delta$-суб\-г\-а\-р\-м\-о\-н\-и\-ч\-е\-с\-к\-ой    функции  $U\not\equiv \pm \infty$ на шаре  $\overline B(R)$
 радиуса $R>r$  
\begin{subequations}\label{Uhdd}
\begin{align}
\int_E U^+\dd \uplambda_{\CC} &
\leq 8 \frac{R+r}{R-r} {\boldsymbol  T}_U(r,R)\, \uplambda_{\CC}(E)
\ln\frac{\pi e^{3}r^2}{\uplambda_{\CC}(E)}\quad\text{на\/ $\CC$ при ${\tt d=2}$},
\tag{\ref{Uhdd}$\CC$}\label{{Uhdd}C}\\
\int_E U^+\dd \uplambda_{\RR^{\tt d}} &\leq 6{\tt d}^{\tt d}
A_{\tt d}(r,R) {\boldsymbol  T}_U(r,R)\Bigl(\uplambda_{\RR^{\tt d}}(E)+\bigl(\uplambda_{\RR^{\tt d}}(E)\bigr)^{2/{\tt d}}\Bigr)
\quad\text{при  ${\tt d}>2$,}
\tag{\ref{Uhdd}{\tt d}}\label{{Uhdd}d}
\end{align}
\end{subequations}  
где  $r$ в ${\boldsymbol   T}_U(r,R)$ 
можно заменить на любое число $r'\in [0,r]$. 
\end{corollary}
%%\begin{remark} Неравенство \eqref{{Uhdd}C}  даже несколько усиливает  неравенство \eqref{inDl+p} теоремы о малых %%плоских множествах из  \S~\ref{prsub} и не только на уровне абсолютных констант. Так,
 %% для  $R:=kr$ с $k>1$  под логарифмом в правой	части неравенства \eqref{inDl+p} можно удалить  параметр $k$, что %%может быть существенным при больших $k$ и  рассмотрении $\delta$-субгармонических и мероморфных функций на %%$\CC$ с очень медленно растущей характеристикой Неванлинны.   
%%\end{remark}

Следующее утверждение --- это пересечение  следствий \ref{cor2} и  \ref{corCRd} с учётом \eqref{Adrog}.   

\begin{corollary}\label{corall2} 
Пусть $U\not\equiv \pm \infty$ ---  $\delta$-субгармоническая функция на всём  $\RR^{\tt d}$, а функция $s\colon \RR^+\to \RR^+\setminus 0$ произвольная. Тогда для любого  $\uplambda_{\RR^{\tt d}}$-измеримого 
подмножества $E\subset \RR^{\tt d}$  при любом $r\in \RR^+$ 
\begin{align*}
\int_{E\cap\overline D(r)} U^+\dd \uplambda_{\CC} &
\leq 8 
\Bigl(1+\frac{2r}{s(r)}\Bigr) {\boldsymbol  T}_U\bigl(r,r+s(r)\bigr)\, \uplambda_{\CC}\bigl(E\cap\overline D(r)\bigr)
\ln\frac{\pi e^{3}r^2}{\uplambda_{\CC}\bigl(E\cap\overline D(r)\bigr)}
\\
\intertext{при ${\tt d=2}$, т.е. на $\CC$, а при ${\tt d}>2$}
\int_{E\cap \overline B(r)} U^+\dd \uplambda_{\RR^{\tt d}} &\leq 30{\tt d}^{\tt d+1}
\Bigl(1+\frac{2r}{s(r)}\Bigr)^{\tt d-1} 
\bigl(1+s(r)\bigr)^{\tt d-2}
{\boldsymbol  T}_U\bigl(r,r+s(r)\bigr)\\
&\times \Bigl(\uplambda_{\RR^{\tt d}}\bigl(E\cap \overline B(r)\bigr)+\bigl(\uplambda_{\RR^{\tt d}}\bigl(E\cap \overline B(r)\bigr)\bigr)^{2/{\tt d}}\Bigr),
\end{align*}
где  первый аргумент $r$ в ${\boldsymbol   T}_U\bigl(r,r+s(r)\bigr)$ 
можно заменить на любое $r'\in [0,r]$. 
\end{corollary}

\subsection{Случай $({\tt d-1})$-мерной поверхностной меры  в $\RR^{\tt d}$}\label{p7_3}
Здесь лишь в некоторой мере  затрагивается  интегрирование разностей субгармонических функций в $\CC$ по мере длины кривых в $\CC$ или по поверхностным мерам на гиперповерхностях в $\RR^{\tt d}$, поскольку возможны и более  общие следствия для интегрирования $\delta$-субгармонических функций по римановым поверхностям \cite{Chi18} или  по многообразиям фрактальной размерности.  В использовании известных сведений о липшицевых функциях и их взаимосвязях со спрямляемостью и мерами Хаусдорфа   следуем, в основном,   \cite[\S~3.2]{Federer}, \cite{FedererI_II},  \cite[3.3]{EG}, \cite[3.7]{Morgan}, опуская конкретные ссылки. 
\subsubsection{Случай кривой в $\CC$}\label{731} 
Пусть  $O\neq \varnothing $ --- {\it открытое множество\/}  на $\RR$  и   $l\colon O\to 
\CC$ --- {\it липшицева инъективная  функция\/} с {\it постоянной Липшица} 
\begin{equation}\label{Lip}
{\sf Lip}(l):=\sup_{\stackrel{x_1\neq x_2}{x_1,x_2\in O}}\frac{\bigl|l(x_2)-l(x_1)\bigr|}{|x_2-x_1|}\in \RR^+,
\end{equation}
 которую  в этом п.~\ref{731} называем также и {\it липшицевой кривой} (без самопересечений и без концов).  По теореме Радемахера такая функция дифференцируема почти всюду по линейной мере Лебега  $\uplambda_{\RR}$, а для модуля её производной   определена  {\it существенная верхняя грань\/} 
\begin{equation}\label{|s'|}
\|l'\|_{\infty}:=\inf \Bigl\{a\in \RR\Bigm| \uplambda_{\RR}\Bigl(\bigl\{x\in E\bigm|
\bigl|l'(x)\bigr|>a\bigr\}\Bigr)=0 \Bigr\}\leq  \sqrt{2}\,{\sf Lip}(l) \in \RR^+.
\end{equation}
Для борелевского подмножества  $E\subset l(O)$ его {\it длина\/} 
\begin{equation}\label{sigma(E)}
\sigma(E):=\int_{l^{-1}(E)}|l'|\dd \uplambda_{\RR}
\end{equation}
совпадает с  одномерной мерой  Хаусдорфа $1\text{\tiny-}{\mathfrak m}^{\text{\tiny $0$}}(E)$ множества $E$. 
В частности, {\it мера длины\/ $\sigma$ на\/}  $l(O)$ --- регулярная мера Бореля на $\CC$.

Липшицеву кривую $l\colon O\to \CC$ называем {\it билипшицевой,\/}  если 
\begin{equation}\label{bs}
{\sf Lip}(l^{-1}):=
\sup_{\stackrel{z_1\neq z_2}{z_1,z_2\in l(O)}}\frac{\bigl| l^{-1}(z_2)-l^{-1}(z_1)\bigr|}{|z_2-z_1|}=
\sup_{\stackrel{x_1\neq x_2}{x_1,x_2\in O}}\frac{|x_2-x_1|}{\bigl|l(x_2)-l(x_1)\bigr|}\in \RR^+.
\end{equation}

\begin{corollary}\label{cor4} Пусть $l\colon O\to \CC$ ---  билипшицева  кривая и $l(O)\subset \overline D(r)$ для некоторого $r\in \RR^+$.  Тогда  каждая $\delta$-субгармоническая функция $U\not\equiv \pm\infty$ на круге $\overline D(R)$ радиуса $R>r$ суммируема    по  мере длины $\sigma$ на $l(O)$  и для любого борелевского подмножества $E\subset l(O)$ 
\begin{equation}\label{EUs}
\int_{E} U^+\dd \sigma \leq  15\,{\sf Lip}(l)\,{\sf Lip}(l^{-1})\,\frac{R+r}{R-r} {\boldsymbol  T}_U(r,R)\, 
\sigma(E) \ln\frac{\pi e^{2}r}{\sigma(E)}<+\infty.
\end{equation}
\end{corollary}
\begin{proof} 
Согласно \eqref{D}  длина  пересечения круга   $\overline D_z(t)$  с  борелевским множеством  $E\subset l(O)$ при произвольном $z\in \CC$  равна 
\begin{equation*}
\int_{l^{-1}(E\cap \overline D_z(t) )}|l'|\dd \uplambda_{\RR}
\overset{\eqref{|s'|}}{\leq}
\sqrt{2}\,{\sf Lip}(l)
 \uplambda_{\RR}\Bigl(l^{-1}\bigl(\overline D_z(t) \bigr)\Bigr),
\end{equation*}
где ввиду  ${\sf Lip}(l^{-1})\overset{\eqref{bs}}{<}+\infty$  диаметр множества 
$l^{-1}\bigl(\overline D_z(t) \bigr)=l^{-1}\bigl(l(O)\cap \overline D_z(t) \bigr)$
не больше ${\sf Lip}(l^{-1}) \cdot 2t$. 
Отсюда  для любых $z\in \CC$ и $t\in \RR^+$ определённая в \eqref{sigma(E)}  длина  $\sigma \bigl(l(O)\cap \overline D_z(t)\bigr)$ 
попавшей в $\overline D_z(t)$ части $l(O)\cap \overline D_z(t)$   не превышает 
$\sqrt{2}\,{\sf Lip}(l) \,{\sf Lip}(l^{-1}) \cdot 2t$.  Таким образом, для модуля непрерывности ${\sf h}_{\sigma_E}$ сужения 
$\sigma_E:=\sigma\bigm|_E$ меры $\sigma$ на $E$      
имеет место неравенство 
$$
{\sf h}_{\sigma_E}(t)\leq {\sf h}_{\sigma}(t)\leq 2\sqrt{2}\,{\sf Lip}(l) \,{\sf Lip}(l^{-1})\, t,
$$
что означает выполнение условия \eqref{{chrh}hd} теоремы \ref{th5} с 
$$p:=1, \quad b:=2\sqrt{2}\,{\sf Lip}(l) \,{\sf Lip}(l^{-1})
$$ 
для меры  Бореля $\mu:=\sigma_E$.
По теореме \ref{th5} требуемое неравенство \eqref{EUs} --- это расписанное \eqref{{Uhd}C}
с учётом равенства $1\text{\tiny-}{\mathfrak m}^{\text{\tiny $0$}}(E)=\sigma (E)=\sigma_E(E)$ для  $E\subset l(O)$.
\end{proof}
\begin{remark} При  дифференцируемости  $l$ на $O$ по теореме Лагранжа о конечных приращениях  
имеем довольно грубую  оценку 
\begin{equation}\label{l1est}
{\sf Lip}(l^{-1})\overset{\eqref{bs}}{\leq} \sup_{x_1,x_2\in O}\frac{1}{\sqrt{\bigl(\Re l'(x_1)\bigr)^2+\bigl(\Im  l'(x_2)\bigr)^2}}.
\end{equation}
\end{remark}

\begin{remark}
Можно рассмотреть и  кривые $l$ с  самопересечениями,  снимая условие инъективности $l$ и  используя  так называемую  функцию кратности для кривой и формулу площади 
\cite[\S~3.2]{Federer}, \cite[3.3]{EG}, \cite[3.7]{Morgan}.
\end{remark}

Рассмотрим также липшицевы кривые со специальной параметризацией.     

Пусть, по-прежнему, $O\in \RR$ --- открытое множество  на $\RR$,  
$y\colon  O\to \RR$ --- {\it липшицева функция\/}  с постоянной Липшица ${\sf Lip}(y)\in \RR^+$. 
Соответствующую липшицеву кривую (без самопересечений и без концов)
\begin{equation}\label{ly}
l_y\colon x\underset{x\in O}{\longmapsto}x+iy(x)\in \CC
\end{equation}  
в $\CC$ с очевидной  постоянной Липшица 
\begin{equation}\label{Lly}
{\sf Lip}(l_y)=\sqrt{1+\bigl({\sf Lip}(y)\bigr)^2}\in \RR^+
\end{equation}
называют {\it  кривой ограниченного наклона $q:={\sf Lip}(y)$ в $\CC$,\/} 
часто несколько некорректно рассматривая её как образ $l_y(O)\subset \CC$ или 
 график $\bigl\{x+iy(x)\bigm| x\in O\bigr\}\subset \CC$. Липшицева кривая вида \eqref{ly} автоматически билипшицева, поскольку
\begin{equation}\label{LipL}
{\sf Lip}(l_y^{-1})\overset{\eqref{bs}}{=}
\sup_{\stackrel{x_1\neq x_2}{x_1,x_2\in O}}\frac{|x_2-x_1|}{\sqrt{(x_2-x_1)^2+\bigl(y(x_2)-y(x_1)\bigr)^2}}
\leq 1.
\end{equation}

\begin{corollary}\label{corarc} Пусть $l_y$ из \eqref{ly}  ---  кривая  ограниченного наклона $q\in \RR^+$ в $\CC$ с мерой длины $\sigma$, а $s\colon \RR^+\to  \RR^+\setminus 0$ ---  произвольная функция. 
Тогда  каждая $\delta$-субгармоническая функция $U\not\equiv \pm\infty$ на  $\CC$
 локально суммируема  по  мере длины $\sigma$ на $l_y(O)$  и для любого борелевского $E\subset l_y(O)$ 
при любом $r\in \RR^+$
\begin{equation*}
\int_{E\cap  D(r)} U^+\dd \sigma \leq 
  15\sqrt{1+q^2}\Bigl(1+\frac{2r}{s(r)}\Bigr) {\boldsymbol  T}_U\bigl(r,r+s(r)\bigr)\, 
\sigma(E\cap \overline D(r)) \ln\frac{\pi e^{2}r}{\sigma(E\cap \overline D(r))}.
\end{equation*}
\end{corollary}
\begin{proof} Рассмотрим открытое множество $O_r :=l_y^{-1}\bigl(l(O)\cap D(r)\bigr)$ и 
билипшицеву кривую $l_y^r\colon O_r\to D(r)\subset \overline D(r)$, равную сужению билипшицевой кривой 
 $l_y$ на $O_r$.  По предыдущему следствию \ref{cor4}, применённому к билипшицевой кривой
$l_y^r$, из неравенства \eqref{EUs} с $R:=r+s(r)$, учитывая  \eqref{Lly} и \eqref{LipL}, 
получаем требуемое в следствии \ref{corarc} неравенство. 

\end{proof}
\subsubsection{Случай гиперповерхности  в\/ $\RR^{\tt d}$}\label{732}
Пусть $O\neq \varnothing$ ---   {\it открытое\/}  подмножество в $\RR^{\tt d-1}$ и 
$s\colon O\to \RR^{\tt d}$ --- {\it липшицева инъективная функция\/} с конечной постоянной Липшица 
${\sf Lip}(l)$, определенной как в \eqref{Lip}.
Такую функций в этом п.~\ref{732} называем также и {\it липшицевой гиперповерхностью в $\RR^{\tt d}$} (без самопересечений и без края).   По теореме Радемахера  для почти всех  точек  $x\in O$ по мере Лебега  $\uplambda_{\RR^{\tt d-1}}$ 
 определён {\it модуль   якобиана\/} $ |Jl|(x)$, равный  арифметическому квадратному корню из суммы квадратов ${\tt d}$ всех миноров  порядка ${\tt d}-1$ матрицы Якоби в точке $x$.  
В частности, почти всюду на $O$ по мере $\uplambda_{\RR^{\tt d-1}}$  корректно определена {\it существенная верхняя грань\/} модуля якобиана
\begin{equation}\label{JlL}
\|Jl\|_{\infty}\leq \sqrt{({\tt d}-1)!\,{\tt d}}\bigl({\sf Lip}(l)\bigr)^{\tt d-1}\in \RR^+.
\end{equation}
Для борелевского подмножества    $E\subset l(O)$ его  {\it $({\tt d-1})$-мерная  площадь}  
\begin{equation}\label{D}
\sigma(E):=\int_{l^{-1}(E)} |Jl|\dd \uplambda_{\RR^{\tt d-1}}
\end{equation}
совпадает с $({\tt d}-1)$-мерной мерой Хаусдорфа $({\tt d}-1)\text{\tiny-}{\mathfrak m}^{\text{\tiny $0$}}(E)$ множества $E$. В частности, {\it мера площади  $\sigma$ на $l(O)$\/} --- регулярная мера Бореля на $\RR^{\tt d}$.

Липшицеву гиперповерхность  называем {\it билипшицевой,\/}  если конечна 
величина ${\sf Lip}(l^{-1})$, определённая в \eqref{bs}.

\begin{corollary}\label{cor5} Пусть $l 
\colon O\to \RR^{\tt d}$ --- билипшицева гиперповерхность в $\RR^{\tt d}$ 
и $l(O)\subset \overline D(r)$ для некоторого\/ $r\in \RR^+$.  Тогда
каждая $\delta$-субгармоническая функция $U\not\equiv\pm\infty$
на шаре $\overline B(R)\subset \RR^{\tt d}$ радиуса $R>r$ суммируема  по  мере площади $\sigma$ на $l(O)$  и для любого борелевского $E\subset l(O)$ выполнено неравенство
\begin{equation}\label{EUs+}
\int_{E} U^+\dd \sigma \leq  
3{\tt d}^{2\tt d} \bigl({\sf Lip}(l)\,{\sf Lip}(l^{-1})\bigr)^{\tt d-1}
A_{\tt d}(R,r)  {\boldsymbol  T}_U(r,R) 
\Bigl( \sigma(E)+\bigl(\sigma(E)\bigr)^{\frac{1}{\tt d-1}}\Bigr).
\end{equation}
\end{corollary}

\begin{proof} Площадь пересечения  $\overline B_x(t)$  с  бо\-р\-е\-л\-евским множеством  $E\subset l(O)$ при произвольном $x\in \RR^{\tt d}$     согласно  \eqref{D} 
равна 
\begin{equation}\label{AS}
\int_{l^{-1}(E\cap \overline B_x(t) )}|Jl|\dd \uplambda_{\RR^{\tt d-1}}
\overset{\eqref{JlL}}{\leq}
\sqrt{({\tt d}-1)!\,{\tt d}}\bigl({\sf Lip}(l)\bigr)^{\tt d-1}
 \uplambda_{\RR^{\tt d-1}}\Bigl(l^{-1}\bigl(\overline B_x(t) \bigr)\Bigr),
\end{equation}
где ввиду  ${\sf Lip}(l^{-1})<+\infty$  диаметр множества 
$l^{-1}\bigl(\overline B_x(t) \bigr)=l^{-1}\bigl(l(O)\cap \overline B_x(t) \bigr)$
не превышает ${\sf Lip}(l^{-1}) \cdot 2t$. Следовательно, множество  $l^{-1}(\overline B_x(t) \bigr)\subset \RR^{\tt d-1}$
содержится в некотором шаре радиуса $2\,{\sf Lip}(l^{-1})\,t$ из $\RR^{\tt d-1}$ и 
$$
 \uplambda_{\RR^{\tt d-1}}\Bigl(l^{-1}\bigl(\overline B_x(t) \bigr)\Bigr)
\leq \frac{\pi^{{\tt d}/2}}{\Gamma({\tt d}/2+1)}\bigl(2\,{\sf Lip}(l^{-1})t\bigr)^{\tt d-1}
\overset{\eqref{volB}}{\leq}6\cdot 2^{\tt d-1}\bigl({\sf Lip}(l^{-1})\bigr)^{\tt d-1}t^{\tt d-1}
$$
при $\tt d>2$ для любых $x\in \RR^{\tt d}$ и $t\in \RR^+$. Согласно \eqref{AS} это означает, что 
для модуля непрерывности  ${\sf h}_{\sigma_E}$ сужения $\sigma_E:=\sigma\bigm|_E$ меры $\sigma$ на $E\subset l(O)$   
\begin{multline*}
{\sf h}_{\sigma_E}(t)\leq {\sf h}_{\sigma}(t)\leq 
6\cdot 2^{\tt d-1}\bigl({\sf Lip}(l^{-1})\bigr)^{\tt d-1}t^{\tt d-1} \,\sqrt{({\tt d}-1)!\,{\tt d}}\bigl({\sf Lip}(l)\bigr)^{\tt d-1}
\\
\leq 3{\tt d}^{\tt d} \bigl({\sf Lip}(l)\,{\sf Lip}(l^{-1})\bigr)^{\tt d-1}
t^{\tt d-1} \quad\text{при всех $t\in \RR^+$ для ${\tt d}>2$.}
\end{multline*}
Отсюда для меры $\mu:=\sigma_E$ выполнено условие   
\eqref{{chrh}hd} теоремы \ref{th5} с 
$$
p:={\tt d-1}, \quad b:=3{\tt d}^{\tt d} \bigl({\sf Lip}(l)\,{\sf Lip}(l^{-1})\bigr)^{\tt d-1}.
$$
По теореме \ref{th5} каждая    $\delta$-суб\-г\-а\-р\-м\-о\-н\-и\-ч\-е\-с\-к\-ая   функция  $U\not\equiv \pm \infty$ на шаре  $\overline B(R)$  радиуса $R>r$  $\sigma_E$-суммируема для любого борелевского $E\subset l(O)$   и по неравенству \eqref{{Uhd}d} с $\mu:=\sigma_E$ получаем
\begin{multline*}
\int_{E} U^+\dd \sigma=\int_{\overline B(r)} U^+\dd \sigma_E \leq 3{\tt d}^{2\tt d} \bigl({\sf Lip}(l)\,{\sf Lip}(l^{-1})\bigr)^{\tt d-1}A_{\tt d}(r,R) {\boldsymbol  T}_U(r,R)\\
\times (\tt d-1)\text{\tiny-}{\mathfrak m}^{\text{\tiny $0$}}(E) \Biggl(1+\frac{1}{\bigl((\tt d-1)\text{\tiny-}{\mathfrak m}^{\text{\tiny $0$}}(E)\bigr)^{\frac{\tt d-2}{\tt d-1}}}\Biggr)\quad\text{для   ${\tt d}>2$.}
\end{multline*}  
Но, как отмечено после \eqref{D}, $\sigma(E)=({\tt d}-1)\text{\tiny-}{\mathfrak m}^{\text{\tiny $0$}}(E)$ и такая замена в правой части последнего неравенства даёт требуемое \eqref{EUs+}.
\end{proof}

\end{fulltext}

\end{document}